\documentclass[12pt]{article}
\newcommand{\subjclass}[2][]{\textbf{Mathematics Subject Classification (#1):} #2}

\usepackage[T1]{fontenc}
\usepackage{graphicx} 

\usepackage{mathtools}
\usepackage{multicol}
\usepackage{units}
\usepackage{faktor}
\usepackage{color}
\usepackage{xcolor}
\usepackage{datetime, color, currfile, bbm, pdfsync}
\usepackage{imakeidx}
\usepackage[subrefformat=parens,labelformat=parens]{subfig}
\makeindex[columns=2, title=Alphabetical Index, intoc, options= -s index-style.ist,columnseprule]

\usepackage{amsmath,amssymb,amsthm,comment}
\usepackage{graphicx}
\usepackage[colorlinks=true, allcolors=black]{hyperref}
\usepackage{enumerate}
\usepackage[intoc,refpage]{nomencl}

\makenomenclature

\usepackage{geometry}
\usepackage[toc,page]{appendix}

\usepackage{tikz-cd,tikz}
\usetikzlibrary{calc,shapes,arrows, positioning, quotes}

\newtheorem{thm}{Theorem}[section]
\newtheorem{prop}[thm]{Proposition}
\newtheorem{lem}[thm]{Lemma}
\newtheorem{cor}[thm]{Corollary}

\newtheorem{notation}{Notation}

\theoremstyle{remark}
\newtheorem{rem}[thm]{Remark}

\newtheorem{exa}[thm]{Example}
\theoremstyle{definition}
\newtheorem{defn}[thm]{Definition}

\newcommand{\Z}{\mathbb{Z}}
\def\K{\mathsf k}
\def\halg{\textnormal{h}_{\textnormal{alg}}}
\def\GK{\textnormal{GKdim}}

\def\F{\mathcal F}
\def\G{\mathcal G}
\def\gr{\operatorname{gr}}
\def\dim{\operatorname{dim}}

\def\H{\operatorname{H}}
\def\span{\operatorname{span}}
\def\Path{\textnormal{Path}}
\def\DIM{\mathsf{DIM}}

\usepackage[
  backend=biber,
  style=numeric,
  sorting=nyt,
  giveninits=true,
  maxbibnames=99,
  doi=true,
  url=false,
  isbn=false,
  eprint=true
]{biblatex}

\renewbibmacro{in:}{%
  \ifentrytype{article}
    {}
    {\printtext{\bibstring{in}\intitlepunct}}%
}

\DeclareFieldFormat
  [article,incollection,inproceedings,book,thesis]
  {title}{\mkbibemph{#1}}

\DeclareFieldFormat[article]{title}{#1}

\DeclareFieldFormat{journaltitle}{\mkbibemph{#1}}
\DeclareFieldFormat{booktitle}{\mkbibemph{#1}}

\DeclareFieldFormat[article]{volume}{\mkbibbold{#1}}

\AtEveryBibitem{%
  \clearlist{language}%
  \clearfield{month}%
  \clearfield{urldate}%
}

\DeclareFieldFormat{doi}{%
  \mkbibacro{DOI}\addcolon\space
  \href{https://doi.org/#1}{\nolinkurl{#1}}%
}

\usepackage{hyperref}
\usepackage{cleveref}

\begin{document}

\title{Growth functions of Algebras and an Application to Leavitt path algebras}

\author{João Schwarz and Alfilgen Sebandal}
\date{}

\maketitle
\begin{abstract}
   In this paper, we consider three growth functions of algebras: the Gelfand-Kirillov dimension, superdimension, and entropy -- as well as a variation of the latter successfully used in the study of Leavitt path algebras. We prove results that make precise the heuristic fact that the Gelfand-Kirillov dimension is suitable for the study of algebras with polynomial growth, the superdimension for algebras with subexponential growth, and the entropy for algebras with exponential growth. We introduce a definition of entropy for graded modules motivated by the entropy for graded algebras. As an application of our study, we give a new criterion for  Leavitt path algebras to be PI. 
\end{abstract}


\section{Introduction}

The notion of growth of algebras is very important, with applications to ring theory (see \cite{Krause_Lenagan:Book2000_Growth:ofAlgebrasandGelfandKirillovDimension, McConnel_Robson:Book2001_NoncommutativeNoetherianRings}), as well to representation theory (\cite{Coutinho:BOOK_AprimerofalgebraicDmodules, Jantzen:Book1983_EinhüllendeAlgebrenhalbeinfacherLie-Algebren}), geometry of Riemaniann manifolds (\cite[Chapter 11]{Krause_Lenagan:Book2000_Growth:ofAlgebrasandGelfandKirillovDimension} and reference therein), among others. The theory of growth of groups is also a classical and important subject \cite{geometricgrouptheory}, and recently there has been study of growth of other kinds of algebraic structures, such as  semigroups, languages,   and algebraic operads \cite{bellzelmanov, Khoroshlin-Piontkovski:2018_Ongenratingseriesoffinitelypresentedoperads, trunctionofunitaryoperads}.

We consider in this paper the growth of associative algebras; particularly, the \emph{Gelfand-Kirillov dimension} ($\GK$) and its \emph{superdimension} analogue ($\DIM$) (cf. \cite{Krause_Lenagan:Book2000_Growth:ofAlgebrasandGelfandKirillovDimension, Borho-Kraft:1976_ÜberdieGelfand-Kirillov-Dimension}), as well as  the \emph{entropy} of $\mathbb{N}$-graded algebras ($\H$) \cite{Newman_Schneider_Shalev:2000_TheEntropyofGradedAlgebras}, and a variation of it, the \emph{algebraic entropy} ($\halg$), introduced for filtered algebras in \cite{Bock_Canto_Barquero_Gonzalez_Campos_Sebandal:2024_AlgebraicEntropyofPathAlgebrasandLeavittPathAlgebrasofFiniteGraphs} and applied to the study of \emph{Leavitt path algebras} of finite graphs.

We make precise the remark made in \cite[Introduction]{Newman_Schneider_Shalev:2000_TheEntropyofGradedAlgebras} that the Gelfand-Kirillov dimension is useful for algebras with polynomial growth, its superdimension analogue for algebras with subexponential growth, and the entropy for algebras with exponential growth. For an $\mathbb{N}$-graded algebra $A$, we have that $\H(A)=0$ if and only if $A$ is finite dimensional. We show that as soon as $\H(A)>1$, then it has exponential growth (Theorem \ref{A-has-exponential-growth-2}). In a related result, we show that if $A$ is a filtered algebra and $\halg(A)>0$, then $A$ has exponential growth as well (Theorem \ref{theo:h(A)>0_G(A)exponential}). We generalize the behavior noted in \cite{Bock_Canto_Barquero_Gonzalez_Campos_Sebandal:2024_AlgebraicEntropyofPathAlgebrasandLeavittPathAlgebrasofFiniteGraphs} for Leavitt path algebras to any affine finite-dimensional filtered algebra $A$: either $\GK (A) < \infty$ and $\halg(A)=0$, or $\halg(A)>0$ and $A$ has exponential growth.
In particular, for finite-dimensional filtrations, there is no subexponential growth (however, if the finite-dimensional assumption is dropped, it is well known that subexponential growth exists, for instance, for enveloping algebras; cf. \cite{Smith:1976_UniversalEnvelopingAlgebraswithSubexponentialbutnotPolynomiallyBoundedGrowth}).

An affine unital algebra $A$ has exponential growth if and only if $\DIM (A)=1$. We prove that if $\DIM (A) <1$, $\H(A) =0$ or $1$ (Proposition \ref{proposition-superdimension}); as soon as $\H(A)>1$, necessarily $\DIM (A) =1$ (Corollary \ref{cor:H(A)>1_DIM(A)=1}). However, we can have different algebras with Gelfand-Kirillov superdimension equals to one, but different entropy. Hence the latter cannot distinguish different kinds of exponential growth.

As a consequence of our study, we obtain a new necessary and equivalent condition for a Leavitt path algebra to be a \emph{PI-algebra} in Theorem \ref{pi:new-criterion} (cf. \cite{Bell_Lenagan_Rangaswamy:2016_Leavittpathalgebrassatisfyingapolynomialidentity}).

Finally, we extend the definitions of the entropy and algebraic entropy to graded modules, and obtain results that are reminiscent of basic facts about the Gelfand-Kirillov dimension. Our main result is a dichotomy theorem for the Gelfand-Kirillov dimension and entropy of a graded module (Theorem \ref{entropy-modules-dichotomy}).

All our algebras are associative and unital, over an arbitray base field $\K$. $\otimes$ means $\otimes_\K$. All modules are left modules.

\section{Preliminaries}

\begin{defn}[Graded Algebra] \label{def:grading_algebra}
 Let $\Gamma$ be a monoid. An algebra $A$ is said to be \emph{$\Gamma$-graded} if $A=\oplus_{\gamma\in \Gamma}A_\gamma$, where each $A_\gamma$ is a subspace of $A$ and 
$A_\gamma A_\delta\subseteq A_{\gamma \delta}$ for all $\gamma,\delta \in \Gamma$. The subspace $A_\gamma$ is called a $\gamma$-$homogenous$ $component$ of $A$.  The elements of $\bigcup_{\gamma \in \Gamma} A_\gamma$ are called $homogenous $ $ elements$ of $A$. The nonzero elements of $A_\gamma$ are called $homogenous$ $of$ $degree$ $\gamma$ and we write deg$(a)=\gamma$ for $a\in A_\gamma\setminus \{0\}$. 

We say a grading for $A$ is \emph{finite-dimensional} if each homogenous component is finite dimensional and say that $A$ is a \emph{finitely $\Gamma$-graded algebra}

\end{defn}

\begin{defn}
We say an algebra $A$ is \emph{affine} if $A$ is finitely generated as an algebra.\end{defn}

\begin{rem}
    If $A$ is an affine algebra and \footnote{$\mathbb{N}=\{0,1,2,\dots\}$}$\mathbb{N}$-graded, then $A$ is finitely $\mathbb{N}$-graded.

\end{rem}

\emph{From now on,} all of our graded algebras will be finitely $\mathbb{N}$-graded.

\begin{defn}[Entropy of a Graded Algebra] \cite{Newman_Schneider_Shalev:2000_TheEntropyofGradedAlgebras}  Let  $A=\bigoplus_{n=0}^\infty A_n$ be a graded algebra. The \emph{entropy} of $A$ is defined as
\[
\H(A)=\limsup_{n\rightarrow \infty} \sqrt[n]{\dim (A_n)}.
\]
\end{defn}

\begin{rem}\label{H=0}
The algebra $A$ is finite dimensional if and only if $\H(A)=0$. 
\end{rem}

We are going to use frequently the remark bellow, without further comments.

\begin{rem}\label{radius-of-convergence} Let $A$ be a finitely algebra
and $f(z)=\sum_{i=0}^\infty a_iz^i$, where $a_i= \dim A_i$. Then $\H(A)=\frac{1}{R}$, where $R$ is the \emph{radius of convergence} of the complex function $f(z)$.   
\end{rem}

\begin{notation}\label{nota:V^m}
    Let $A$ be an algebra over a field $\K$ and $V$ a subspace of $A$. We denote $V^0:=\K$   and $V^m:=
\operatorname{span}\{ a_{1}a_{2}\dots a_{m}~|~ a_{j}\in V \}$ for every $m>0$.  
\end{notation}

\begin{defn}[Filtered Algebra] \label{def:filtration}
    An  algebra $A$ is said to be \emph{filtered} if it is endowed with a collection of subspaces $\mathcal{F}=\{V_n\}_{n=0}^\infty$ such that

\begin{enumerate}
    \item [(i)] $0=V_0\subset V_1\subset \dots \subset V_n\subset V_{n+1}\subset \dots\subset  A$,
    \item [(ii)] $\displaystyle\bigcup_{n\geq 0} V_n= A$, 
    \item [(iii)] $V_nV_m\subseteq V_{n+m}.$
\end{enumerate}
An algebra $A$ with a filtration $\mathcal{F}$ is denoted by $(A, \mathcal{F})$. If each $V_i$ is finite dimensional, the filtration $\F$ is called \textit{finite dimensional}. If each quotient space $V_i/V_{i-1}$ is finite dimensional, we say that $\F$ is \textit{quotient-finite dimensional}
. If $V_1$ contains the unit and $V_n=V_1^n$ for each $n\geq 0$, $\mathcal{F}$ is said to be \emph{standard}. 
\end{defn}

\begin{rem}
    If $A$ is affine, then it has a finite-dimensional  standard filtration. 
\end{rem}

\begin{defn}[Graded Algebra Associated to a Filtered Algebra]

Let $(A, \mathcal{F})$ be a filtered algebra with the filtration $\mathcal{F}=\{ V_n \}_{n\geq 0}$. The graded algebra associated with $(A, \mathcal{F})$ is defined to be \[
\operatorname{gr}_\mathcal{F}(A)=\bigoplus_{n\geq 0}{V_{n}/V_{n-1}},\]
with multiplication $(x+ V_{n-1})(y+V_{m-1}):= xy + V_{n+m-1}$ and the convention that $V_{-1}=0$.

\end{defn}

\begin{defn}[Algebraic Entropy of a Filtered Algebra] \cite{Bock_Canto_Barquero_Gonzalez_Campos_Sebandal:2024_AlgebraicEntropyofPathAlgebrasandLeavittPathAlgebrasofFiniteGraphs}
Let $(A, \mathcal{F})$ be a filtered algebra with a quotient-finite dimensional filtration $\mathcal{F}=\{ V_n\}_{n\geq 0}$. The \emph{algebraic entropy} of $(A, \mathcal{F})$ is defined to be 
\[
\halg(A,\mathcal{F}):= \begin{cases}
    0,  &\text{if }A \text{ is finite dimensional},\\
    \displaystyle\limsup_{n\rightarrow \infty}  \frac{\log\dim(V_n/V_{n-1})}{n}  & \text{otherwise}.
\end{cases}
\]
If the context is clear, for a filtered algebra $(A,\mathcal{F})$, we simply write $\halg(A,\mathcal{F})= \halg(A)$.
\end{defn}

\begin{rem}\label{entropy-bridges}
Let $A$ be an algebra with filtration $\F=\{V_n\}_{n \geq 0}$,  $\operatorname{gr}_\F(A)$ its associated graded algebra and $\halg(A, \mathcal{F})>0$. Then $\halg(A, \mathcal{F})=\log \H(\operatorname{gr}_\mathcal{F}(A))$. Thus, $ \H(\operatorname{gr}_\mathcal{F}(A))=e^{\halg(A, \mathcal{F})}>1$. Indeed in \cite{Bock_Canto_Barquero_Gonzalez_Campos_Sebandal:2024_AlgebraicEntropyofPathAlgebrasandLeavittPathAlgebrasofFiniteGraphs}, the algebraic entropy of a filtered algebra is defined by taking the logarithm of the entropy of the corresponding graded algebra in order to compare with the \emph{Gelfand-Kirillov dimension}; see Definition \ref{def:GKdim}.
\end{rem}

\begin{notation}
Let $P$ be a property defined on $\mathbb{N}$. If there exists $n_0 \in \mathbb{N}$ such that $P(n)$ holds for all $n \geq n_0$, we say $P(n)$ holds \emph{for $n \gg 0$} (read: ``for all sufficiently large $n$'').
\end{notation}

\begin{defn}[Growth] 
Let $\Phi$ be the set of functions
 $f: \mathbb{N} \longrightarrow \mathbb{R}$ such for $n\gg0$, 
\[f(n)>0\quad\text{and} \quad f(n+1) \geq f(n).\]
For two functions  $f,g\in \Phi$, we write $f \leq^* g$ if there are constants $c, m \in \mathbb{N}$ such that $f(n) \leq cg(mn)$ for all $n\in \mathbb{N}$. We write $f \sim g$ if $f \leq^* g$ and $g \leq^* f$. It can be shown that $\sim$ is an equivalence relation and we represent the class of $f\in \Phi$ in $\Phi/\sim$ by $\mathcal{G}(f)$, called the \emph{growth of $f$}. Finally, the relation induced by $ \leq^* $ in $\Phi/\sim$ shall be denoted by $\leq$.
\end{defn}
\begin{notation} Let 
\footnote{$\mathbb{R}^{\geq 0}$ is the set of all non-negative real numbers and $\mathbb{R}^{> 0}:=\mathbb{R}^{\geq 0}\setminus \{0\}$.}$y \in \mathbb{R}^{\geq 0}$ and $x\in \mathbb{R}^{>0}$. The growth of the function $n \mapsto n^y$ is denoted $\mathcal{P}_y$ and the growth of the function $n \mapsto e^{n^x}$ is denoted by $\mathcal{E}_x$.
\end{notation}

\begin{rem}
    If $f(n)=\ln(n+1)$, then    $\mathcal{G}(f) > \mathcal{P}_0$ but $\mathcal{G}(f) < \mathcal{P}_\varepsilon$ for any $\varepsilon >0$. So there is no analogue of the Archimedean property for this order. Finally, we remark that not every pair $\mathcal{G}(f)$, $\mathcal{G}(h)$ is comparable. Here, by $\mathcal{G}(f) >\mathcal{G}(g)$, we mean that $f\geq^*g$ is due to the existence of $c,m\in \mathbb{N}$ for which $f(n)>cg(mn)$.
\end{rem}

\section{Algebraic Entropy and GK-Dimensions}

\begin{notation}\label{asymptotic}
For $f \in \Phi$,
\[\gamma(f)=  \limsup_{n\rightarrow \infty} \frac{\log  f(n)}{\log n} = \limsup_{n \to \infty}\big(\log_n \, f(n)\big).\]
\end{notation}

\begin{prop}\label{growth}
    Let $f,g \in \Phi$.
    \begin{enumerate}[{\rm (i)}]
        \item If $\mathcal{G}(f)=\mathcal{G}(g)$, then $\gamma(f)=\gamma(g)$.
        \item $\gamma(f+g)=\sup \{ \gamma(f), \gamma(g) \}$.
        \item If $f(n)=p(n)$ for $n\gg 0$ and $p(x) \in \mathbb{R}[x]$, then $\gamma(f)=\deg \, p$.
        \item If $\gamma(f)=\lim_{n \to \infty} \log_n \, f(n)$, $\gamma(fg)= \gamma(f)+\gamma(g)$.
        \item If $g(n) \leq f(an+b), \, n\gg 0$, where $a,b \in \mathbb{N}$, then $\gamma(g) \leq \gamma(f)$.
    \end{enumerate}
\end{prop}

\begin{proof}
    \cite[Lemma 2.1b]{Krause_Lenagan:Book2000_Growth:ofAlgebrasandGelfandKirillovDimension}, \cite[Lemma 8.1.7]{McConnel_Robson:Book2001_NoncommutativeNoetherianRings}
\end{proof}

\begin{defn}[Gelfand-Kirillov Dimension] \label{def:GKdim}
Let $A$ be an affine algebra and $V$ a finite dimensional generating space for $A$. Let $d_V(n)=\dim \, \left (\sum_{i=0}^n V^i\right )$. Then $\mathcal{G}(d_V)$ is independent of the choice of the generating space $V$ (\cite[Lemma 1.1]{Krause_Lenagan:Book2000_Growth:ofAlgebrasandGelfandKirillovDimension}). Hence,
\[ \gamma(d_V)=\limsup_{n \rightarrow \infty} \big(\log_n(d_V(n))\big),\]
is also independent of $V$, so by Proposition \ref{growth}(i) this number is well-defined and is called the \emph{Gelfand-Kirillov dimension} of $A$, denoted $\GK \, A$. We also put $\mathcal{G}(A)=\mathcal{G}(d_V)$, where $V$ is any choice of finite dimensional generating space for $A$, and call it the \emph{growth of $A$}.
\end{defn}

\begin{defn}[Growth of an algebra]
    Let $A$ be an affine algebra. We say $A$ has 
    \begin{itemize}
        \item \emph{polynomial growth} if $\mathcal{G}(A)=\mathcal{P}_m$ for some natural $m$;
        \item \emph{exponential growth} if $\mathcal{G}(A)=\mathcal{E}_1$;
        \item \emph{subexponential growth} or \emph{intermediate growth} if $\mathcal{G}(A)< \mathcal{E}_1$ but $\mathcal{G}(A) \nleq \mathcal{P}_m$ for all $ m \in \mathbb{N}$.
    \end{itemize}
\end{defn}

It is clear that if $A$ is an affine algebra with $\G(A)=\mathcal{E}_x$ for $x \in (0,1)$, then $A$ has subexponential growth.

\begin{defn}
    Let $A=\bigoplus_{n=0}^\infty A_n$ be an affine graded algebra. Then the formal series  $P_A(t)=\sum_{n=1}^\infty a_n t^n \in\Z[[t]]$ where $a_n=\dim (A_n)$, is the \emph{Poincaré series} of $A$. $\Z[[t]] \hookrightarrow \mathbb{Q}[[t]]$, and we have an embedding $\iota: \mathbb{Q}(t) \hookrightarrow \mathbb{Q}[[t]]$. $P_A(t)$ is said to be \emph{rational} if there exists $q(t) \in \mathbb{Q}(t)$ such that $\iota(q)=P_A$.
\end{defn}

\begin{thm}\label{rational-Poincaré-series}
    If $A$ is an infinite-dimensional affine graded algebra such that $P_A(t)$ is rational, letting $R$ be its radius of convergence, we have two possibilities:

    (a) $R<1$, and $A$ has exponential growth.

    (b) $R=1$, and there is an integer $m \in \mathbb{N}$ with $\GK A=m$.
\end{thm}
\begin{proof}
\cite[Chapter III, Section 1]{Lorenz:Book1988_Gelfand-KirillovdimensionandPoincareseries}.
\end{proof}

\begin{cor}\label{cor:H(A)=1_iff_GKA<infty}
    If $A$ is an infinite-dimensional affine graded algebra such that $P_A(t)$ is rational, then $\H(A)\geq 1$. In particular, $\H(A)=1$ if and only if $\GK (A)\in \mathbb{N}\setminus \{0\}$ and $\H(A)>1$ if and only if $A$ has exponential growth, which happens if and only if $\GK A=\infty$.

\end{cor}

\begin{thm}\label{GKdim-filtered}\textnormal{\cite[Proposition 6.6]{Krause_Lenagan:Book2000_Growth:ofAlgebrasandGelfandKirillovDimension}}
Let $(A, \mathcal{F})$ be a filtered algebra with a finite dimensional filtration $\mathcal{F}=\{ V_n\}_{n\geq 0}$. Then $\G(A)=\G(\operatorname{gr}_\F A)=\G(\dim V_n)$ and
\[
\GK (A) =\limsup_{n\rightarrow \infty} \log_n\dim V_n. 
 \]

\end{thm}

\begin{rem}
For an affine filtered algebra $(A, \mathcal{F})$ with a finite dimensional filtration $\mathcal{F}=\{ V_n\}_{n\geq 0}$,  $\halg(A,\mathcal{F})\leq \GK(A)$.
\end{rem}

\begin{thm}\label{no-finite-GK}
Let $A$ be an affine algebra with finite dimensional filtration. If $\halg(A) > 0$, then $\GK(A) = \infty$.
\end{thm}
\begin{proof}
    Let $\F=\{V_n\}_{n \geq 0}$ be a finite dimensional filtration for $A$. We have $\dim V_n \geq \dim (V_n/V_{n-1})$. Suppose we had $\GK(A)=\kappa<\infty$. The by \cite[Lemma 2.1a)]{Krause_Lenagan:Book2000_Growth:ofAlgebrasandGelfandKirillovDimension} and Theorem \ref{GKdim-filtered} there is an $n_0 \in \mathbb{N}$ such that for all $n \geq n_0$, $\dim(V_n) \leq n^\kappa$. Hence
    \[ \halg(A) = \limsup_{n \to \infty} \frac{\log(\dim(V_n/V_{n-1}))}{n} \leq \limsup_{n \to \infty} \kappa \frac{\log n}{n}=0.\]

    This leads to a contradiction. Hence $\GK(A)=\infty$.
\end{proof}

\begin{rem} \label{rem:halg(A)=log~m}
Let $A=\K\langle x_1, \ldots, x_m \rangle$ be the free algebra in $m$ generators. Then $A$ is $\mathbb{N}$-graded given by the length of words, that is, $A=\bigoplus_{i=0}^\infty V^i$ where $V=\operatorname{span}\{x_1, x_2, \dots, x_m\}$; see Notation \ref{nota:V^m}. Thus,  $\H(\operatorname{gr}_\mathcal{F}(A))=m$ where $\mathcal{F}$ is the natural filtration given by the grading in $A$. Thus,  $\halg(A)=\log(H(\operatorname{gr}_\mathcal{F}(A)))= \log m$. 
\end{rem}

\begin{defn}
    Let $A=\bigoplus_{i=0}^\infty A_i$ be an $\mathbb{N}$-graded algebra. The \emph{natural filtration} for $A$ is $\F=\{A_{(n)}\}_{n \geq 0}$ where $A_{(n)}=\bigoplus_{i=0}^n A_i$. 
\end{defn}

For the (Leavitt) path algebra associated to a finite graph, the following bound is weaker than \cite[Proposition 4.5 and Corollary 4.6]{Bock_Canto_Barquero_Gonzalez_Campos_Sebandal:2024_AlgebraicEntropyofPathAlgebrasandLeavittPathAlgebrasofFiniteGraphs} but generalizes to an arbitrary affine algebra.  See Definition \ref{def:LPA} for the definition of Leavitt path algebra and Definition \ref{def:LPAstandardfiltration} for its standard filtration considered in \cite{Bock_Canto_Barquero_Gonzalez_Campos_Sebandal:2024_AlgebraicEntropyofPathAlgebrasandLeavittPathAlgebrasofFiniteGraphs}. 

\begin{cor}
Let $(A,\F)$ be an affine algebra with a finite dimensional standard filtration $\F =\{V_n\}_{n\geq 0}$ where $V_1$ is an $m$-dimensional vector space. Then $\halg(A) \leq \log(m)$.
\end{cor}
\begin{proof}
    By Remark \ref{rem:halg(A)=log~m} and \cite[Lemma 4.1]{Bock_Canto_Barquero_Gonzalez_Campos_Sebandal:2024_AlgebraicEntropyofPathAlgebrasandLeavittPathAlgebrasofFiniteGraphs}.
\end{proof}

The following is an improvement of a result proven in \cite{Smith:1976_UniversalEnvelopingAlgebraswithSubexponentialbutnotPolynomiallyBoundedGrowth}. We follow ideas in \cite{Lorenz:Book1988_Gelfand-KirillovdimensionandPoincareseries}.


\begin{thm}\label{convergence-theorem} 
Let $A=\bigoplus_{i=0}^\infty A_i$ be an affine finitely $\mathbb{N}$-graded algebra and consider its natural filtration $\F=\{A_{(n)}\}_{n \geq 0}$. Then  $\lim_{n \to \infty} (\dim A_{(n)})^{1/n}$ exists.
\end{thm}
\begin{proof}
For some $c \in \mathbb{N}$, $A_{(c)}$ generates $A$ as an algebra. We claim that for each $m,n$,
\[ (\dagger) \qquad  A_{(m+n)} \subset \sum_{i=0}^{c-1} A_{(m+i)} A_{(n-i)}.\]
We prove this by induction in $m$, the case $m=0$ being clear. The problem is reduced to showing that $(\ddagger) \, A_{m+n} \subset \sum_{i=0}^{c-1} A_{m+i} A_{n-i}$. Let us fix algebra generators  $x_1, \ldots, x_s$, $\operatorname{deg} x_i \leq c$. Let $\mu=x_{i_1} \ldots x_{i_v}$ be a nonzero element of degree $m+n$. Pick $u \leq v$ minimal with respect that $\nu=x_{i_1}\ldots x_{i_u}$ having degree greater or equal than $m$. Then
\[ m \leq \operatorname{deg} \nu \leq m +\operatorname{deg}(x_{i_u})-1 \leq m+c-1.\]

So $\mu \in A_{\operatorname{deg}(\nu)} A_{m+n-\operatorname{deg}(\nu)}$,  which proves $(\ddagger)$, and hence $(\dagger)$. So $\dim A_{(n+m)} \leq \dim A_{(n+c)} \dim A^m$, for all $m ,n \in \mathbb{N}$, and so $\dim A_{(kn)} \leq (\dim A_{(n+c)})^k$ for all $n,k \in \mathbb{N}$.

Now fix $m$, and choose $n$ large enough such that, with $k=\lfloor \frac{n}{m} \rfloor+1$, $\dim A_{(n)} \leq \dim A_{(mk)} \leq (\dim A_{{(m+c)}})^k \leq (\dim A_{(m+c)})^{(\frac{n}{m}+1)}$.

Hence $\limsup_{n \to \infty} (\dim A_{(n)})^{1/n} \leq \limsup_{n \to \infty}(\dim A_{(m+c)})^{\frac{1}{m}+\frac{1}{n}}= (\dim A_{(m+c))})^{\frac{1}{m}} $.

Hence $\limsup_{n \to \infty} (\dim A_{(n)})^{1/n} \leq \liminf_{m \to \infty} (\dim A_{(m+c)})^{1/m}=\liminf_{m \to \infty} (\dim A_{(m)})^{1/m} $.

Hence $\lim_{n \to \infty} (\dim A_{(n)})^{1/n}$ exists.
\end{proof}

The following is an important result as a consequence, which strengthens Theorem \ref{no-finite-GK}.

\begin{thm}\label{theo:h(A)>0_G(A)exponential}
    Let $A$ be an affine algebra with a finite dimensional filtration. If $\halg(A) > 0$, then $\G(A)$ is exponential.
\end{thm}
\begin{proof} Let $\F=\{ V_n \}_{n \geq 0}$ be a finite dimensional filtration for $A$.
By Remark \ref{entropy-bridges}, \linebreak    $\H(\operatorname{gr}_\F(A))>1$, and by definition, $\dim (\operatorname{gr}_\F A)_{(n)} =\dim (\oplus_{i=0}^n V_{i}/V_{i-1})=\dim V_{n+1}$. We have that  $\lim_{n \to \infty} (\dim V_ n)^{1/n}$ exists by Theorem \ref{convergence-theorem} and is equal to $\H(\operatorname{gr}_\F(A))>1$. Hence there is an $C>1$ such that for bigger enough $n$, $\dim V_n > C^n$. By Theorem \ref{GKdim-filtered} and since $A$ is affine, it follows that $A$ has exponential growth.
\end{proof}

In particular, we have the following:

\begin{cor}\label{no-intermediate}
Let $A$ be an affine algebra with a finite dimensional filtration. If $\halg(A) > 0$, then $A$ cannot have subexponential growth.    
\end{cor}

\begin{rem}
 Theorem \ref{theo:h(A)>0_G(A)exponential} generalizes the Case 3 trichotomy  in 
\cite[Corollary~5.17]{Bock_Canto_Barquero_Gonzalez_Campos_Sebandal:2024_AlgebraicEntropyofPathAlgebrasandLeavittPathAlgebrasofFiniteGraphs} for (Leavitt) path algebras over finite graphs to arbitrary affine infinite-dimensional algebras.
\end{rem}

The following Proposition will generalize \cite[Proposition 3.4]{Bock_Canto_Barquero_Gonzalez_Campos_Sebandal:2024_AlgebraicEntropyofPathAlgebrasandLeavittPathAlgebrasofFiniteGraphs}.

\begin{prop}
\label{prop:Finitefiltration_h(A,F)=0_iff_h(A,G)=0}
   Let $\mathcal{F}=\{ V_n \}_{n\geq 0}$ and  $\mathcal{G}= \{ W_n \}_{n\geq 0}$ be two finite dimensional filtrations of an affine algebra $A$ such that $W_n \subset W_1^n$. Then $\halg(A, \mathcal{F})=0$ implies that $\halg(A, \mathcal{G})=0$.
\end{prop}
\begin{proof} Let $\mathcal{F}=\{ V_n \}_{n\geq 0}$ and  $\mathcal{G}= \{ W_n \}_{n\geq 0}$. 
    Since $A$ is affine, there exists  $p \in \mathbb{N}^+$ such that $W_1 \subset V_p$. Hence $W_n \subset V_{pn}$ for all $n \in \mathbb{N}$. Consider the filtration $\mathcal{H}=\{ Z_n \}_{n\geq 0}$ of $A$, where $Z_n=V_{pn}$. By \cite[Proposition 3.4(1)(2)]{Bock_Canto_Barquero_Gonzalez_Campos_Sebandal:2024_AlgebraicEntropyofPathAlgebrasandLeavittPathAlgebrasofFiniteGraphs}, $\halg(A, \mathcal{H})=0$, and since $W_n \subset Z_n$ for all $n$, $\halg(A, \mathcal{G})=0$, again by \cite[Proposition 3.4(1)]{Bock_Canto_Barquero_Gonzalez_Campos_Sebandal:2024_AlgebraicEntropyofPathAlgebrasandLeavittPathAlgebrasofFiniteGraphs}.
\end{proof}


\begin{defn}\label{definition-PI-algebra}
    Let $X=\{x_1,x_2,\ldots\}$ be a countable set, let $\Phi$ be an associative commutative unital ring, and let $\Phi\langle X\rangle$ denote the free associative unital $\Phi$-algebra generated by $X$. Let
$
f(x_1,\ldots,x_n)\in\Phi\langle X\rangle
$
be a noncommutative polynomial involving only the variables $x_1,\ldots,x_n$, and assume that at least one monomial of maximal degree has coefficient $1$.\footnote{This assumption excludes trivial identities such as the polynomial $px$, which vanishes identically in every ring of characteristic $p$.} A $\Phi$-algebra $A$ is called a \emph{PI-algebra} if there exists such a polynomial $f$ satisfying
$
f(a_1,\ldots,a_n)=0
$
for all $a_1,\ldots,a_n\in A$. In this case, $f$ is called a \emph{polynomial identity} of $A$. If $\Phi=\mathbb{Z}$, we simply say that $A$ is a \emph{PI-ring}.
\end{defn}

\begin{rem}\label{rem:dim(V_n/V_n-1)bounded--GK(A)leq1}
Let $A$ be an affine algebra with a finite-dimensional filtration $\mathcal{F}=\{V_n\}_{n\geq 0}$. By the proof of \cite[Lemma 3.7]{Bock_Canto_Barquero_Gonzalez_Campos_Sebandal:2024_AlgebraicEntropyofPathAlgebrasandLeavittPathAlgebrasofFiniteGraphs}, if the sequence $\{\dim (V_n/V_{n-1})\}_{n\geq 0}$ is bounded, then $\GK (A)\leq 1$.     
\end{rem}

\begin{prop}\label{prop:Pi-algebra} 
    Let $A$ be an affine algebra with a finite  dimensional filtration $\mathcal{F}=\{ V_n \}_{n\geq 0}$. If the sequence $\{ \dim(V_n/V_{n-1}) \}_{n\geq 0}$ is bounded, then $A$ is a PI-algebra. 
\end{prop}
\begin{proof}
  By Remark~\ref{rem:dim(V_n/V_n-1)bounded--GK(A)leq1}, $\GK (A)=0$ or $1$. If $\GK( A) =0$, then $A$ is trivially PI (for instance, if $\dim A=k$, then $A$ satisfies the standard identity in $k+1$ indeterminates). If $\GK (A) =1$, it is PI by the Small-Stafford-Warfield Theorem \cite{Small_Stafford_Warfield:1985_AffinealgebrasofGelfand-KirillovdimensiononearePI} .
\end{proof}

\begin{defn} \label{def:directedgraph}
    A \emph{directed graph} is a tuple $
E=(E^0,E^1,s,r)$,
where $E^0$ and $E^1$ are disjoint sets and $s,r:E^1\to E^0$ are maps. The elements of $E^0$ and $E^1$ are called \emph{vertices} and \emph{edges}, respectively. For $e\in E^1$, $s(e)$ is called the \emph{source} of $e$ and we say $s(e)$ \emph{emits} $e$, and $r(e)$ is called the \emph{range} of $e$. A (finite) \emph{path}  is a sequence of edges $
p=e_1e_2\cdots e_n$
with $r(e_i)=s(e_{i+1})$ for each $1\leq i\leq n-1$. Here, $n$ is the \emph{length} of $p$ and is denoted by $l(p)$. The vertices are regarded as paths of length $0$. The source of $p$ is $s(p)=s(e_1)$ and its range is $r(p)=r(e_n)$. We set $s(v)=r(v)=v$ for every vertex $v$. The set of paths in $E$ is denoted by $\textnormal{Path}(E)$. If $p=e_1\cdots e_n$ has length $n>0$ with $s(e_1)=r(e_n)$ and $s(e_i)\neq s(e_j)$ for each $i\neq j$, then $p$ is called a \emph{cycle based at $s(p)$}. An edge $f$ for which $s(f)=s(e_i)$ for some $i=1,2,\dots, n$ and $f\neq e_j$ for all $j=1,2,\dots,n$ is called an \emph{exit} of the cycle $p$.
\end{defn}

\begin{defn} \cite{abramsaramolinabook}\label{def:LPA} 
    For a graph $E=(E^0, E^1, s,r)$ and a field $\K$, the \emph{Leavitt path algebras} of $E$, denoted by $L_\K(E)$, is the algebra generated by the sets $\{v:v\in E^0\}$, $\{ e :e \in E^1  \}$ and $\{ e^* : e \in E^1  \}$ with coefficients in $\K$, subject to the relations
\begin{enumerate}[nolistsep]
\item[\textnormal{(i)}]
$v_iv_j=\delta_{i,j}v_i$ for every $v_i, v_j\in E^0$;
\item[\textnormal{(ii)}] $s(e)e=e = e r(e)$ and $r(e)e^*=e^*=e^*s(e)$ for all $e \in E^1$;

\item[\textnormal{(CK1)}] $e^*e'=\delta_{e,e'}r(e)$ for all $e, e'\in E^1$;

\item[\textnormal{(CK2)}]$\sum_{ \{e \in E^1 :s(e)=v    \}  } e e^*=v$ for every $v\in E^0$ which is not a sink.

\end{enumerate}
\end{defn}
\noindent In \cite{Bock_Canto_Barquero_Gonzalez_Campos_Sebandal:2024_AlgebraicEntropyofPathAlgebrasandLeavittPathAlgebrasofFiniteGraphs}, Leavitt path algebras are endowed with its \emph{standard filtration}:

\begin{defn}[Standard Filtration for Leavitt path algebras]\label{def:LPAstandardfiltration}
    Let $E=(E^0,E^1, s,r)$ be a finite graph and $L_\K(E)$ its associated Leavitt path algebra. The \emph{standard filtration} of $L_\K(E)$ is $\mathcal{L}=\{W_n\}_{n\geq 0}$ where $W_n=\span \{pq^*~|~p,q\in \Path(E),~r(p)=r(q),~
    l(p)+l(q)\leq n\}$ for each $n\geq 0$. We again note that here, vertices are considered paths of length $0$. 
    
\end{defn}

\begin{rem} 
For a finite graph $E$, the standard filtration $\mathcal{L}=\{W_n\}_{n\geq 0}$ for $L_\K(E)$ is finite and $W_n=W_1^n$. Note that $W_0 = \span E^0\neq \K$ but $W_n=W_1^n$ for all $n\geq 0$. Note that this is not a standard filtration in the sense of Definition \ref{def:filtration}, but the meaning of the expression \emph{standard filtration} will always be clear by context.

\end{rem}

\begin{prop}\label{prop:noexit-bounded}    Let $E$ be a finite graph and $L_\K(E)$ its associated Leavitt path algebra with its standard filtration $\mathcal{L}=\{W_n\}_{n\geq 0}$. 
Then $E$ has no cycle with an exit if and only if $\{\dim(W_n/W_{n-1}) \}_{n\geq 0}$ is bounded. \end{prop}

\begin{proof} $(\Leftarrow)$ 
Thus directly follows from \cite[Lemma 3.7]{Bock_Canto_Barquero_Gonzalez_Campos_Sebandal:2024_AlgebraicEntropyofPathAlgebrasandLeavittPathAlgebrasofFiniteGraphs} and \cite[Theorem 4.1]{Bell_Lenagan_Rangaswamy:2016_Leavittpathalgebrassatisfyingapolynomialidentity}. For completeness: Let $c$ be a cycle of length $m$ based at a vertex $v$ and let $e$ be an exit from $v$. For each $r\geq 0$, let $p_r=c^re$. Then $l(p_r)=rm+1$. Now, consider the elements $p_rp_s^*$. Since $c$ has an exit at $e$, we have at $|s^{-1}(v)|\geq 2$ and hence, $ee^*\neq s(v)$. It follows that $p_rp_s^*$,  is a nonzero and linearly independent element $W_m/W_{m-1}$ where $m=l(p_r)+l(p_s)$. 

    Fix $k\geq 0$ and set $n=km+2$. For each $r=0,1,\dots, k$, take $s-k=r$. Then 
    \[
    l(p_r)+l(p_s)=rm+1+sm+1=(r+s)m+2=km+2=n.
    \]
    Thus, the $k+1$ elements  $p_0p_k^*, p_2p_{k-1}^*, \dots, p_kp_0^*$ all are distinct, linearly independent elements in $W_n/W_{n-1}$. Thus, $\dim(W_n/W_{n-1})\geq k+1$. Hence, the sequence $\{\dim(W_n/W_{n-1})\}_{n\geq 1}$ is unbounded. 

\noindent $(\Rightarrow)$   
Suppose that no cycle in $E$ has an exit and let $c_1,c_2, \dots , c_l$ be the cycles in $E$.

Now consider a cycle $c_i=e_1e_2\dots e_k$ and an element $yc_i^sq$ for some paths $y,p\in \textnormal{Path}(\widehat{E})$. By (CK1), we may assume that $y\in\textnormal{Path}(E)$. If $q\in \textnormal{Path}(E)$, then we may rewrite $yc_i^sq=y'c_i^{s'}$. Suppose $q$ is a ghost path, that is, $q=g_1^*g_2^*\dots g_l^*$ for some $g_i\in E^1$ with $r(c_i)=s(q)=s(g_1^*)=r(g_1)$. If $g_1=e_k$, then  
\[yc_i^sq=yc_i^{r-1}e_1\dots e_{k-1}g_1g_1^*g_2^*\dots g_l^*=yc_i^{s-1}e_1\dots e_{k-1}s(g_1)g_2^*\dots g_l^*=yc_i^{s-1}e_1\dots e_{k-1}g_2^*\dots g_l^*.\]
Thus, $yc_i^sq=y''c_i'^{s-1}q'$ where $c_i'$ is a cyclic permutation of $c_i$ and $q'$ is a ghost path. We may also adopt the same arguments for elements $yc_i^sq$ where $s<0$ with the convention 
$c^s:=(c^*)^{-s}$.

Hence, for a large enough positive integer $m$, we may let $W_m/W_{m-1}$ be a space that contains only non-zero monomials of the form $xc_i^rp^*$ for some $r\in \mathbb{Z}\setminus \{0\}$ and $x,p\in \text{Path}(E)$ where $x,p\neq c$ (here, vertices are considered paths). Since $c_i$ has no exit, $p$ is not a subpath of $c_i$ by the (CK2) relation. That is, $r(p)$ is a vertex in $c_i$ and $s(p)$ is not a vertex in $c_i$.

In particular, $xc_i^rp^*\in W_m/W_{m-1}$ where $m=l(x)+l(c_i)|r|+l(p)$. Now fixing the triple $(x,i,p)$, and solving $l(x)+l(c_i)|r|+l(p)=m$, we obtain only one value of $r$ and taking both $\pm r$, we have $2$ basis elements in $W_m/W_{m-1}$. Let $M_1=\{(x,i,p)~|~x,p\in \textnormal{Path}(E), i\leq l, r(x)\textnormal{~and~}r(p) \textnormal{~are vertices in~}C_i\}$ and $N_1=|M_1|<\infty $. Then $\dim(W_{m'}/W_{m'-1})\leq 2N_1$ for all $m'\geq m$. Take $N_2=\max\{\dim(W_n/W_{n-1})~|~n<m\}<\infty$ (since $E$ is finite). It follows that the sequence $\{\dim(W_n/W_{n-1})\}_{n\geq 0}$ is bounded by $ \max\{2N_1, N_2\}$. \end{proof}

\begin{thm}\label{pi:new-criterion}
Let $E$ be a finite graph and $L_\K(E)$ its associated Leavitt path algebra equipped with the standard filtration $\mathcal{L}=\{W_n\}_{n\geq 0}$. Then $L_\K(E)$ is a PI-algebra if and only if sequence $\{ \dim(W_n/W_{n-1}) \}_{n\geq 0}$ is bounded.
\end{thm}

\begin{proof} Follows from Proposition \ref{prop:Pi-algebra}, \cite[Theorem 4.1]{Bell_Lenagan_Rangaswamy:2016_Leavittpathalgebrassatisfyingapolynomialidentity}, and Proposition~\ref{prop:noexit-bounded}. 
\end{proof}

Recall that if $A$ and $B$ are two affine $\K$-algebras with  finite dimensional filtrations $\mathcal{F}=\{ V_n \}_{n\geq 0}$ and $\mathcal{G}=\{W_n \}_{n\geq 0}$, respectively, then $A \otimes B$ has a finite dimensional filtration $\mathcal{H}=\{Z_n\}_{n\geq 0}$, where $Z_n = \sum_{i+j = n} V_i \otimes W_j$. If two algebras $R$ and $S$ are $\mathbb{N}$-graded, $R=\bigoplus_{n=0}^\infty R_n$, $S=\bigoplus_{n=0}^\infty S_n$, $R \otimes S$ is also $\mathbb{N}$-graded, with the homogeneous component of degree $n$ being $\sum_{i+j=n} R_i \otimes S_j$. Clearly, then, if we take  $R=\operatorname{gr}_\mathcal{F}(A), S=\operatorname{gr}_\mathcal{G} (B)$, then $R \otimes S \simeq \operatorname{gr}_\mathcal{H} (A \otimes B)$ as $\mathbb{N}$-graded algebras with the operations above.

\begin{prop}
    Let $A$ and $B$ be 
two affine $\K$-algebras with finite dimensional filtrations. Then $\halg(A \otimes B) \leq \max\{ \halg(A), \halg(B)\}$ (compare with \cite[Lemma 3.10]{Krause_Lenagan:Book2000_Growth:ofAlgebrasandGelfandKirillovDimension}).
\end{prop}
\begin{proof} 
 Let $A$ and $B$ be 
two affine $\K$-algebras with finite dimensional filtrations $\mathcal{F}$ and $\mathcal{G}$, respectively. Let $R=\gr_\mathcal{F}(A)=\bigoplus_{n\geq 0 }R_n$ and $S=\gr_\mathcal{G}(B)=\bigoplus_{n\geq 0 }S_n$ and
let $f(z)=\sum_{n=0}^\infty( \dim R_n)z^n$ and $g(z)=\sum_{n=0}^\infty (\dim S_n)z^n$. Then $f(z)$ and $g(z)$ has radius of convergence $1/\H(R)$ and $1/\H(S)$, respectively. Then, by Merten's Theorem in complex analysis, $\H(R \otimes S)$, the inverse radius of convergence of the Cauchy product of $f(z)g(z)$, is greater than or equal to $\min \{ 1/\H(R), 1/\H(S) \}$. So $\H( R \otimes S) \leq \max \{ \H(R), \H(S) \}$. Taking $\log$, in view of the discussion preceeding this Proposition combined with Remark \ref{entropy-bridges}, we have our result.
\end{proof}

\begin{defn}
    Two quotient-finite dimensional filtrations $\mathcal{F}=\{ V_n \}_{n\geq 0}$ and $\mathcal{G}=\{ W_n \}_{n\geq 0}$ of an affine algebra $A$ are said to be \emph{strongly equivalent}, denoted by $\mathcal{F} \sim \mathcal{G}$, if there exist $k, \ell \in \mathbb{N}$ such that for every $n$, $\dim (V_n/V_{n-1}) \leq \dim (W_{n+k}/{W_{n+k-1}})$ and $\dim (W_n/W_{n-1}) \leq \dim (V_{n+\ell}/{V_{n+\ell-1}})$.
\end{defn}

\begin{prop}
    Let $\mathcal{F}=\{ V_n \}_{n\geq 0}$ and $\mathcal{G}=\{ W_n \}_{n\geq 0}$ be two equivalent quotient-finite dimensional filtrations of an affine algebra $A$.
    Let $R=\operatorname{gr}_\mathcal{F} (A)$, $S=\operatorname{gr}_\mathcal{G} (A)$, and suppose that both are strongly graded. Then $\halg(A, \mathcal{F})=\halg(A, \mathcal{G})$.
\end{prop}
 \begin{proof}
     Let $R=\operatorname{gr}_\mathcal{F} (A)$, $S=\operatorname{gr}_\mathcal{G} (A)$. Remember that if $f(z)$ is the complex function given by $f(z)=\sum_{n=0}^\infty \dim R_n z^n$, $\H(R)$ is equal the inverse of the radius of convergence of $f$. Clearly, the same relation holds for $S$ and let $g(z)=\sum_{n=0}^\infty \dim S_n z^n$. Then
 \[\H(R)=\limsup_{n \to \infty} \sqrt[n]{\dim R_n} \leq \limsup_{n \to \infty} \sqrt[n]{\dim S_{n+k
}} \leq \limsup_{n \to \infty} \sqrt[n]{ n(n-1)\cdots(n-k)(\dim S_{n+k})}.\] The number \[
\limsup_{n \to \infty} \sqrt[n]{ n(n-1)\cdots(n-k)(\dim S_{n+k})}\] is the inverse of the radius of convergence of $g^{(k)}(z)$, the $k$-th derivative of $g(z)$, which is the same as that of $g(z)$. Hence 
\[\limsup_{n \to \infty} \sqrt[n]{ n(n-1)\cdots(n-k)(\dim S_{n+k})}=\limsup_{n \to \infty} \sqrt[n]{\dim S_n},\] and so $\H(R) \leq \H(S)$. By symmetry, $\H(S) \leq \H(R)$, and so they are equal. Hence, by Remark \ref{entropy-bridges}, $\halg(A,\mathcal{F})=\halg(A,\mathcal{G})$.
 \end{proof}

\section{Definition of Entropy for Modules}

\begin{defn}
    Let $A$ be a graded algebra. A \emph{grading} on an $A$-module $M$ is a collection of subspaces $\{M_i\}_{i\geq 0}$ such that $M=\bigoplus_{i=0}^\infty M_i$ and $V_i M_j \subset M_{i+j}$ for all $i,j\geq 0$. A \emph{graded $A$-module} is an $A$-module $M$ together with a grading on $M$. If each $M_i$ is finite dimensional, we say that the grading is \emph{finite dimensional} and say $M$ is a \emph{finitely graded module}.  
   
\end{defn}
\emph{From now on}, all of our graded modules will be finitely graded.

\begin{defn}
   
    Let $A$ be a filtered algebra with a filtration $\mathcal{F}=\{V_i\}_{i\geq 0}$. An $A$-module $M$ is said to be \emph{filtered} if it is endowed with a collection of subspaces $\Omega=\{ M_i \}_{i\geq 0}$ such that 
    \begin{enumerate}
    \item [(i)] $0=M_0\subset M_1\subset \dots \subset M_n\subset M_{n+1}\subset \dots\subset  M$,
    \item [(ii)] $\displaystyle\bigcup_{i\geq 0} M_i= M$, 
    \item [(iii)] $V_iM_j\subseteq M_{i+j}.$
\end{enumerate}
    
   A module $M$ with a filtration $\Omega$ is denoted by $(M,\Omega)$.    If each $M_i$ is finite-dimensional, the filtration $\Omega$ is called a \emph{finite-dimensional filtration}. We denote by $\gr_\Omega( M)$ the \emph{associated graded module to  $M$}, which is a module over $\gr_\mathcal{F}(A)$. As a vector space,
    \[ \gr_\Omega(M) = \bigoplus_{i=0}^\infty \big(M_i/M_{i-1}\big), \]
    with the convention that $M_{-1}=0$. The action of $\gr_\mathcal{F}(A)$ on $\gr_\Omega(M)$ is given by \[(a+ V_{i-1})(m+M_{j-1}):=(am +M_{i+j-1}).\]
\end{defn}

\begin{rem} Clearly,
    a filtered algebra $A$ is a filtered module over itself where the action is defined by the algebra multiplication. Moreover, the associated graded algebra to $A$ is precisely the associated graded module to $A$ as a filtered module.   
\end{rem}

\begin{defn}[Entropy of a Graded Module] Let  $A=\bigoplus_{n=0}^\infty A_n$ be a graded algebra. The \emph{entropy} of a graded $A$-module  $M=\bigoplus_{n=0}^\infty M_n$ is defined as
\[
\H(M)=\limsup_{n\rightarrow \infty} \sqrt[n]{\dim (M_n)}.
\]
\end{defn}

Clearly $M$ is finite dimensional if and only if $\H(M)$=0. Our first result shows that there is a gap between the values that $\H(M)$ can assume if $M$ is infinite dimensional.

\begin{prop}\label{entropy-gap}
    Let $M$ be an infinite-dimensional graded module $M$. Then $\H(M) \geq 1$.
\end{prop}
\begin{proof}
    If $M$ is infinite dimensional, for infinitely many $n$, $\dim M_n \geq 1$, and so for infinitely many $n$, $\sqrt[n]{(\dim M_n)} \geq 1$. Hence $\H(M) = \limsup_{n \to \infty} \sqrt[n]{(\dim M_n)} \geq 1$.
\end{proof}

We will make use of the following handy lemma:

\begin{lem}\label{lem:limsup_n-throot(an+bn)_leq_max(a,b)}
    Let $\{ a_n \}_{n\geq 0}$ and $\{ b_ n \}_{n\geq 0}$ be two sequences of positive real numbers. If $\limsup_{n \to \infty} \sqrt[n]{a_n}=a$ and $\limsup_{n \to \infty} \sqrt[n]{b_n}=b$, then
    \[ \limsup_{n \to \infty} \sqrt[n]{a_n+b_n}= \max\{a,b\}. \]
\end{lem}
\begin{proof}
Since $a_n \leq a_n + b_n $ and $b_n \leq a_n + b_n$, we have $\limsup_{n \to \infty} \sqrt[n]{a_n+b_n} \geq \max\{a,b\}$. On the other hand, $a_n+b_n \leq 2 \max \{ a_n , b_n \}$, so $\sqrt[n]{a_n+b_n} \leq \sqrt[n]{2} \max \{ \sqrt[n]{a_n}, \sqrt[n]{b_n})$. Taking the limit superior on both sides of this inequality give us $\limsup_{n \to \infty} \sqrt[n]{a_n+b_n} \leq \max\{a, b\}.$
\end{proof}

The following lemma is clear. 
\begin{lem}\label{lem:H_epi_mono}
Let $M$ and $N$ be graded $A$-modules. If there is a graded epimorphism of modules $M\rightarrow N$, then $\H(M)\geq \H(N)$. If there is a graded monomorphism $M\rightarrow N$, then $\H(M)\leq \H(N)$.   
\end{lem}

Now we are going to see three Propositions in which the algebraic entropy behaves like the Gelfand-Kirillov dimension (cf. \cite[Proposition 5.1]{Krause_Lenagan:Book2000_Growth:ofAlgebrasandGelfandKirillovDimension}). 

\begin{prop}\label{prop:H(oplus Mi)=max H(M_i)}
    Let $M_1, M_2,\ldots ,M_n$ be finitely $\mathbb{N}$-graded $A$-modules for a finitely $\mathbb{N}$-graded algebra $A$. Then
    for $\mathcal{M}=\bigoplus_{i=1}^n M_i$ or $\sum_{i=1}^n M_i$, $\H(\mathcal{M})= \max \{ \H(M_i) \}_{i=1}^n$.
\end{prop}
\begin{proof} The homogeneous component of degree $n$ of the module  $\mathcal{M}=\bigoplus_{i=1}^n M_i$ is the direct sum of the homogeneous components of degree $n$ of $M_i$; similarly for $\mathcal{M}=\sum_{i=1}^n M_i$. For $\mathcal{M}=\bigoplus_{i=1}^n M_i$, the conclusion is  a direct consequence of Lemma \ref{lem:limsup_n-throot(an+bn)_leq_max(a,b)}.
 We have a graded epimorphism $\bigoplus_{i=1}^n M_i \rightarrow \sum_{i=1}^n M_i$, so $\max \{ \H(M_i) \}_{i=1}^n \geq \H(\sum_{i=1}^n M_i)$ by Lemma \ref{lem:H_epi_mono}. The reverse inequality is obvious. So we have our result.
\end{proof}

\begin{prop}\label{like-GK-2}
    Let $A$ be a graded algebra and let $M$ be a finitely generated graded $A$-module. Suppose that $\dim A_{n+1} \geq \dim A_n$ for every $n$. Then $\H(M) \leq  \H(A)$.
\end{prop}

\begin{proof}
    Let $M$ be generated by homogeneous elements $m_1, \ldots, m_\ell$ of homogeneous degrees $d_1, \ldots, d_\ell$. Let $A(-d)$ be the \emph{module} $A$ with shift on its grading: $A(-d)_n=A_{n-d}$ -- we set $A_{n-d}=0$ if $ n<d$. We have a graded epimorphism of $A$-modules $\bigoplus_{i=1}^\ell A(-d_i) \rightarrow M$ given by $a_i\mapsto a_im_i$ for each $a_i\in A(-d_i)$. For each $n$, $\dim M_n \leq \sum_{i=1}^\ell \dim A(-d_i)_n \leq \ell \dim A_n$. So $\H(M) \leq \limsup_{n \to \infty} \sqrt[n]{\ell \dim A_n} = \H(A)$.
\end{proof}

\begin{defn}
    Let $M$ be a finitely generated left $A$-module, where $A$ is an affine algebra. Choose a finite-dimensional space $F \subset M$ that generates it as a module, and let $V$ be a frame for $A$. Let $d_{F,V}(n)=\dim \, V^n F$. The \emph{Gelfand-Kirillov dimension} of $M$ is
    \[ \GK( M)=\GK_A( M)=\limsup_{n \to \infty} \, \log_n \, d_{F,V}(n), \]
    and the value is independent of the choices of $F, V$ which can proved by an argument similar to \cite[Proof of Lemma 1.1]{Krause_Lenagan:Book2000_Growth:ofAlgebrasandGelfandKirillovDimension}.
\end{defn}

\begin{rem}
    For a finitely generated  graded $A$-module $M$, $M$ is finite dimensional if and only if $\H(M)=0$ if and only if $\GK(M)=0$. 
\end{rem}

\begin{thm}[$\GK$--Entropy Dichotomy] \label{entropy-modules-dichotomy}
    Let $M$ be a finitely generated infinite-dimensional graded $A$-module for an affine graded algebra $A$.
    
    (a) If $\H(M) > 1$, then $\GK(M)=\infty$.


    (b) If $\GK(M) \in (0, \infty)$, then $\H(M)=1$. 
    
\end{thm}
\begin{proof}
We consider the sequences $\{b_n\}_{n\geq 0}$ and $\{c_n\}_{n\geq0}$ where $b_n=\dim M_n$ and $c_n= \dim \bigoplus_{i=0}^n M_i$. By \cite[Lemma 6.1(b)]{Krause_Lenagan:Book2000_Growth:ofAlgebrasandGelfandKirillovDimension}, $\GK(M)= \limsup  \log_n c_n$,

(a) By assumption, for infinite values of $n$, $\sqrt[n]{b_n}>C$ for some $C > 1$, and since $c_n \geq b_n$, the same holds for $c_n$. But for infinitely many $n$, 
$\log_nc_n>\frac{\log C^n}{\log n}=\frac{n\log C}{\log n}$. Since $\lim_{n \to \infty} \frac{n}{\log n} = \infty$, $\GK(M)= \limsup \log_n c_n= \infty$.


(b)  By (a), it follows that $\H(M)\leq 1$, and by Proposition \ref{entropy-gap}, $\H(M)\geq 1$. Hence, $\H(M)=1$.

\end{proof}

\begin{thm}\label{A-has-exponential-growth-2}
    If $A$ is an affine graded algebra, then $\H(A)>1$ implies that $A$ has exponential growth. Hence, if $\GK(A) < \infty$, $\H(A)=1$ or $0$.
\end{thm}
\begin{proof}
Note that clearly $A$ is a finitely generated graded module over itself. The claim follows from the proof of Theorem \ref{entropy-modules-dichotomy}(a). \end{proof}




In an analogous way to the definition of algebraic entropy of a filtered algebra in \cite{Bock_Canto_Barquero_Gonzalez_Campos_Sebandal:2024_AlgebraicEntropyofPathAlgebrasandLeavittPathAlgebrasofFiniteGraphs}, we also have

\begin{defn}[Algebraic Entropy of a Filtered Module] Let $(A, \mathcal{F})$ be a filtered algebra with a quotient-finite dimensional filtration $\mathcal{F}=\{ V_n\}_{n\geq 0}$. Let $M$ be an $A$-module with a quotient-finite dimensional filtration $\Omega=\{ M_n \}_{n\geq 0}$. 
The \emph{algebraic entropy} of $(M, \Omega)$ is defined to be
\[
\halg(M, \Omega):= \begin{cases}
    0,  & \text{if }M \text{ is finite dimensional},\\
    \displaystyle\limsup_{n\rightarrow \infty}  \frac{\log\dim(M_n/M_{n-1})}{n},  & \text{otherwise}.
\end{cases}
\]
If the context is clear, for a filtered module $(M, \Omega)$, we simply write $\halg(M, \Omega)= \halg(M)$.
\end{defn}

It is clear that if $(M, \Omega)$ is a quotient-finite dimensional filtered module over a quotient-finite filtered algebra $(A,\F)$, $\halg(M, \Omega)=\log H(\gr_\Omega (M)) $ (cf. Remark \ref{entropy-bridges}). In particular, if $\halg(M, \Omega)>0$, then $H(\gr_\Omega (M))  > 1$, and so, by Theorem \ref{entropy-modules-dichotomy}(a):

\begin{cor}
  Let $(M, \Omega)$ be a finitely generated filtered $A$-module, where $(A, \F)$ is a quotient-finite dimensional filtered algebra and $\Omega$ is a quotient-finite dimensional filtration on $M$.  If $\halg(M,\Omega)>0$ then $\GK(M) = \infty$.
\end{cor}

We can translate our other results for $\H(M)$ to results about $\halg(M)$.

\begin{thm}\label{theo:h(+Mi)=maxh(M_i)}
If $\{(M_i, \Omega_i)\}_{i=1}^n$ is a collection of quotient-finite dimensional filtered modules over a quotient-finite dimensional affine filtered algebra $(A, \F)$. Then we have the following:
\begin{enumerate}
    \item 
[(a)]
$\halg(\bigoplus_{i=1}^n M_i)=\max \{\halg(M_i)\}_{i=1}^n$.

\item [(b)] If $\dim(\gr_\F(A)_{n})\leq \dim(\gr_\F(A)_{n+1})$ for every $n$, 
 then $\halg(A,\F) \geq \halg(M_i, \Omega_i)$ for each $i=1,2,\dots, n$.
 \end{enumerate}
\end{thm}
\begin{proof} Since  $\log \H(\gr_{\mathcal{F}}(A))=\halg(A, \F)$ and  $\log \H(\gr_{{\Omega_i}}(M_i))= \halg(M_i, \Omega_i)$, the results follow from Propositions \ref{prop:H(oplus Mi)=max H(M_i)} and \ref{like-GK-2} applied to  $\gr_{\mathcal{F}}(A)$ and  $ \gr_{{\Omega_i}}(M_i)$. 
\end{proof}

\section{ Gelfand-Kirillov superdimension}

\begin{defn}[Gelfand-Kirillov superdimension, \cite{Borho-Kraft:1976_ÜberdieGelfand-Kirillov-Dimension}]
Let $A$ be an affine algebra and $V$ be a finite-dimensional vector space containing the unit $1$. Let $d_V(n):=\dim V^n$. Then the \emph{Gelfand-Kirillov superdimension} of $A$ is
\[\DIM(A):= \limsup_{n\to\infty} \frac{\log \log d_V(n)}{\log n}.\]
Its value is independent of the choice of $V$.
\end{defn}

By \cite[$\S$ 2.3, 2.16]{Borho-Kraft:1976_ÜberdieGelfand-Kirillov-Dimension}, for any affine algebra $A$, $0 \leq \DIM  (A) \leq 1$. Moreover, for any $0\leq r \leq 1$, there exists an affine algebra $B$ such that $\DIM (B)=r$.

Now, let us recall a useful lemma for the Gelfand-Kirillov superdimension.  

\begin{lem}[{\cite[Lemma 1.5(b)]{Borho-Kraft:1976_ÜberdieGelfand-Kirillov-Dimension}\label{lemma-superdimension}}]
    Let $f \in \Phi$. Then 
    \[\displaystyle \limsup_{n\to\infty} \frac{\log \log f(n)}{\log n} = \inf\{\varepsilon \in \mathbb{R} \mid f(n) \le e^{n^\varepsilon} \textnormal{ for } n\gg 0 \}.\]
    
\end{lem}

Note that this lemma shows that an affine algebra $A$ has an exponential growth if and only if $\DIM(A)=1$.

We are going to see, in the next example, how $\DIM(\cdot)$ is a finer invariant than $\GK(\cdot)$ for algebras with subexponential growth (see also the discussion in \cite[$\oint$ 12.1]{Krause_Lenagan:Book2000_Growth:ofAlgebrasandGelfandKirillovDimension}).
\begin{exa}\label{Smith}
    In \cite{Smith:1976_UniversalEnvelopingAlgebraswithSubexponentialbutnotPolynomiallyBoundedGrowth}, it is shown that if $\mathfrak{g}$ is a finitely generated Lie algebra over a field of characteristic $0$, then $\GK(U(\mathfrak{g}))<\infty $  if and only if $\dim(\mathfrak{g})<\infty $, and that the growth of $U(\mathfrak{g})$ is subexponential when so is the growth of $\mathfrak{g}$. 
    
    As an example, let $\mathfrak{L}$ be the infinite-dimensional Lie algebra with basis $\{ x, y_1, y_2, \ldots \}$ and brackets given by $[x, y_i]=y_{i+1}, \, [y_i, y_j]=0$. Let $A=U(\mathfrak{L})$. Then $\mathcal{G}(A)=\mathcal{G}\big(\exp(\sqrt{n})\big)$ and the growth is subexponential. $\GK(U(\mathfrak{L}))=\infty$, so the Gelfand-Kirillov dimension gives us limited information. However, $\DIM(U(\mathfrak{L})) = \frac{1}{2}$. More generally, if $A$ is any affine algebra, $x$ a number in $(0,1)$, if $\G(A)=\mathcal{E}_x$, then $A$ has subexponetial growth, $\GK(A)=\infty$, but $\DIM(A)=x$. Hence, for this class of algebras, $\DIM(\cdot)$ is a finer invariant than $\GK(\cdot)$.
\end{exa}

\begin{prop}\label{proposition-superdimension}
    Let $A$ be an affine graded algebra with  $\DIM (A) < 1$. Then
    \[\H(A)=
    \begin{cases}
    $1$, &     \text{if } $A$ \text{ is infinite dimensional}\\
    $0$, & \text{otherwise.}
    \end{cases}
    \]
    \end{prop}
\begin{proof}
We introduce two functions on the natural numbers: $g(n)=\dim A_n$ and $f(n)= \dim \bigoplus_{i=0}^n A_i$. By \cite[Lemma 6.1(b)]{Krause_Lenagan:Book2000_Growth:ofAlgebrasandGelfandKirillovDimension}, and Lemma \ref{lemma-superdimension}, $\DIM (A)= \inf\{\varepsilon \in \mathbb{R} \mid f(n) \le e^{n^\varepsilon} \text{ for } n \gg 0\}$. Let $c = \DIM (A)$, and let $\varepsilon>0$ small enough such that $c + \varepsilon < 1$. Then for $n \gg 0$, $g(n) \leq f(n) \leq e^{n^{c+\varepsilon}}$. Hence 
\[\H(A) = \limsup_{n \to \infty} \sqrt[n]{g(n)} \leq \limsup_{n \to \infty} (e^{n^{c +\varepsilon})^{\frac{1}{n}}} =\limsup_{n \to \infty} e^{n^{c +\varepsilon -1}}.\] Let $d= c + \varepsilon -1 < 0$. Then $n^d \to 0$ as $n \to \infty $, and so, since the exponentiation function is continuous, $\limsup_{n \to \infty} e^{n^{c +\varepsilon -1}} \leq e^0=1$. So $\H(A) \leq 1$. Combining Proposition \ref{entropy-gap} and Remark \ref{H=0}, we obtain our result about $\H(A)$.
\end{proof}

\begin{cor}\label{cor:H(A)>1_DIM(A)=1}
    If $A$ is an affine graded algebra with $\H(A)>1$, then $\DIM(A)=1$.
\end{cor}

\begin{exa}
    There are algebras with 
   GK superdimension $1$,  but with different entropies. For instance, let $F_n(\K)$ be the free unital associative algebra in $n$ generators. Then by \cite[$\S$ 2.4]{Borho-Kraft:1976_ÜberdieGelfand-Kirillov-Dimension}, for any $n$, $\DIM (F_n(\K))=1$. However, $\H(F_n(\K))=n$ by \cite[Corollary 3.2]{Newman_Schneider_Shalev:2000_TheEntropyofGradedAlgebras}. Hence entropy is a finer invariant for algebras with exponential growth. 
\end{exa}

\subsection*{Acknowledgments}
The second author acknowledges Dr.~Christopher Bernido of the Research Center for Theoretical Physics (Central Visayan Institute Foundation) and Dr. Håkan Sollervall of Linnaeus University for hosting her during the research fellowships at their respective institutions, as well as Dr. Vyachelsav Futorny of Southern University of Science and Technology and Dr. Ardeline Mary Buhphang of North-Eastern Hill University for the research stays which allowed the completion of this paper.

\printbibliography[
  heading=bibintoc,
  title={References}
]

\bigskip
\bigskip

\noindent\textsc{João Schwarz: Shenzhen International Center for Mathematics, Southern University of Science and Technology, China}

\noindent\textit{Email address:} \texttt{jfschwarz.0791@gmail.com}
\bigskip

\noindent\textsc{Alfilgen Sebandal: Central Visayan Institute Foundation, Philippines and Linnaeus University, Sweden}

\noindent\textit{Email address:} \texttt{a.sebandal@rctpjagna.com}

\end{document}